\documentclass[a4paper,11pt]{amsart} 
\usepackage{MiscStyle}
\usepackage{NumberArtStyle}
\usepackage{ShortcutStyle}
\usepackage{enumitem}
\usepackage{tikz-cd}
\usepackage{verbatim}
\usepackage[applemac]{inputenc}    

\usepackage{color}

\begin{document}

\author{J. Fern\'andez de Bobadilla}
\address{Javier Fern\'andez de Bobadilla:  
(1) IKERBASQUE, Basque Foundation for Science, Maria Diaz de Haro 3, 48013, 
    Bilbao, Basque Country, Spain;
(2) BCAM,  Basque Center for Applied Mathematics, Mazarredo 14, 48009 Bilbao, 
Basque Country, Spain; 
(3) Academic Colaborator at UPV/EHU.} 
\email{jbobadilla@bcamath.org}

\author{G. Pe\~nafort}
\address{G. Pe\~nafort: BCAM,  Basque Center for Applied Mathematics, Mazarredo 14, 48009 Bilbao, 
Basque Country, Spain. } 
\email{gpenafort@bcamath.org}

\author{J. E. Sampaio}
\address{Jos\'e Edson Sampaio: (1) Departamento de Matem\'atica, Universidade Federal do Cear\'a,
	      Rua Campus do Pici, s/n, Bloco 914, Pici, 60440-900, 
	      Fortaleza-CE, Brazil.\newline
		  (2) BCAM - Basque Center for Applied Mathematics,
	      Mazarredo, 14 E48009 Bilbao, Basque Country - Spain. 
}
\email{edsonsampaio@mat.ufc.br}

\title[Milnor fibre of finitely determined map-germs $(\C^2,0)\to(\C^3,0)$]{Topological invariants and Milnor fibre for $\cA$-finite germs $\C^2\to\C^3$}

\thanks{J.F.B. was supported by ERCEA 615655 NMST Consolidator Grant, MINECO by the project 
reference MTM2016-76868-C2-1-P (UCM), by the Basque Government through the BERC 2018-2021 program and Gobierno Vasco Grant IT1094-16, by the Spanish Ministry of Science, Innovation and Universities: BCAM Severo Ochoa accreditation SEV-2017-0718 and by Bolsa Pesquisador Visitante Especial (PVE) - Ciencias sem Fronteiras/CNPq Project number:  401947/2013-0. \\
G.P. and J.E.S. were partially supported by the ERCEA 615655 NMST Consolidator Grant and also by the Basque Government through the BERC 2018-2021 program and Gobierno Vasco Grant IT1094-16, by the Spanish Ministry of Science, Innovation and Universities: BCAM Severo Ochoa accreditation SEV-2017-0718.}

	\begin{abstract}
 This note is the observation that a simple combination of known results shows that the usual analytic invariants of a finitely determined multi-germ $f\colon (\C^2,S)\to(\C^3,0)$ ---namely the image Milnor number $\mu_I$, the number of crosscaps and triple points, $C$ and $T$, and the Milnor number $\mu(\Sigma)$ of the curve of double points in the target--- depend only on the embedded topological type of the image of $f$. As a consequence one obtains the topological invariance of the sign-refined Smale invariant for immersions $j\colon S^3\looparrowright S^5$ associated to finitely determined map germs $(\C^2,0)\to(\C^3,0)$. 

 	\end{abstract}

\maketitle

\section{Introduction}

Let $f\colon (\C^2,S)\to (\C^3,0)$ be an $\cA$-finite multi-germ (for $\cA$-finiteness and related notions, we refer the reader to \cite{MondNuno2019}). We write
\[X=\im f\subseteq (\C^3,0)\]
for the image of $f$.
Associated to $X$, there is the  \emph{Milnor fibre}
\[\F=g^{-1}(\delta)\cap B,\]
where $g=0$ is a reduced equation for $X$, $B$ is a small ball around the origin of $\C^3$ and $\delta$ is nonzero and small enough. 
In the previous version of this work, we related betti numbers $b_i(\F)$ (with coefficients in $\Z_2$) to analytic invariants associated to $\cA$-finite germs, namely the numbers $C, T$ of crosscaps and triple points and the Milnor number $\mu(D)$ of the double point curve in the source of $f$. The main point of finding these relations was that, together with some other relations and the fact that the betti numbers (with any coefficients) of $\F$ are invariants of the embedded topological type of $X$ \cite{Le73}, they implied then that $C, T$, the image Milnor number $\mu_I$ and the Milnor number $\mu(\Sigma)$ of the double point in the target, depend only on the topological type of $X$. As a consequence we proved the topological invariance of the sign-refined Smale invariant for immersions $(\C^2,0)\to(\C^3,0)$ associated with finite map germs.

However, Siersma communicated to us that our expression for $b_1(\F)$ was  incorrect, and pointed us to a paper of van Straten~\cite{vS} we were unaware of, where the correct formula is provided. After replacing our mistaken statement by the correct expression of $b_1(\F)$, the mentioned approach to show the topological invariances becomes hopeless. 

Nevertheless, it turns out that one can combine the Euler characteristic formula (which was obtained before by other authors, as we will mention below), the Marar-Mond formulae, the topological invariance of $C-3T$ (proved by N\'emethi and Pint\'er~\cite{NP}) and a result of L\^e~\cite{Le73} to recover the topological invariances claimed above. 

We shall describe now the $\cA$-invariants of $f$ we are dealing with: The \emph{target double point space} of $f$ is \[\Sigma=\Sing (X).\]
Being a reduced curve with an isolated singularity, $\Sigma$ has a well defined Milnor number \cite{BuchweitzGreuelTheMilnorNumber}.
The \emph{source double point space} \[D=f^{-1}(\Sigma)\] is a disjoint collection of germs of plane curves with isolated singularity \cite{MararMondCorank1}, \cite{Marar2012Double-point-cu}. It is immediate that the Milnor numbers of each of the components of $D$ and $\mu(\Sigma)$ are holomorphic $\cA$-invariants of $f$. We denote by $\mu(D)$ the sum $\sum_{z\in S}\mu(D,z)$. 

 Let $F=(f_s,s)\colon (\C^2\times \C,S)\to (\C^3\times \C,0)$ be a stabilization of $f$, and let $f_s\colon U_s\to \C^3$ be a stable perturbation of $f$. The space
 \[X_s=f_s(U_s)\] is called the \emph{disentanglement} of $X$. The homotopy type of $X_s$ does not depend on the chosen stabilization, and it is that of a wedge of 2-dimensional spheres \cite{Mond:1991}. The number of spheres,
 \[\mu_I=b_2(X_s),\] is a holomorphic $\cA$-invariant of $f$ known as the \emph{image Milnor number of $f$}.
 
  Being the image of a stable mapping, the disentanglement can only exhibit the following three classes of singularities: transverse double points, crosscaps and transverse triple points. Crosscaps and triple points are isolated singularity types and give rise to holomorphic $\cA$-invariants of $f$, namely
\[C=\# \text{ cross-caps in } X_s,\]
\[T=\# \text{ triple points in } X_s.\]

We thank D. Sierma and M. Tibar for their careful reading of the previous version of this paper, during which they spotted the existence of a problem with our computation of $b_1(\F)$, and for pointing us to the papers ~\cite{ST},~\cite{vS}, which we initially missed.  We also thank A. Nemethi for explanations of his work with G. Pint\'er on Smale invariants.

\section{The betti numbers of $\F$ (corrected)}
In Theorem 1.1 of the previous version of this work, we mistakenly claimed that the first betti number  of $\F$, with $\Z_2$ coefficients, behaves as follows:
 \begin{itemize}
  \item [(i)] If $C=0$ then $b_1(\F,\Z_2)=T+1$.
  \item [(ii)] If $C\not=0$ then $b_1(\F,\Z_2)=T$.
 \end{itemize}

To compute $b_1(\F)$, we employed the Homology Splitting method, introduced by Siersma in~\cite{Si1} and~\cite{Si2}. A full explanation of the method, and further related references, can be found in Sierma's survey paper~\cite{Si3}, and in the recent paper by Siersma and Tibar~\cite{ST}, which presents the latest evolution of the method for $1$-dimensional critical set.   
However, Siersma pointed out that $b_1(\F)$ had been computed previously by Van Straten and the result differed from ours. A mistake, related to the local model around triple points, was then spotted  in our computations. The correct statement, a particular case of~\cite{vS}, is as follows:

\begin{prop}Let $f\colon (\C^2,S)\to (\C^3,0)$ be an $\cA$-finite multi-germ. Then 
\[b_1(\F)=\vert S\vert-1.\]
\end{prop}

Having found $b_1(\F)$, calculating $b_2(\F)$ amounts to computing the Euler characteristic of $\F$. This has been previously achieved by Siersma and Massey, see~\cite{Si4}, Section~4 and~\cite{MaSi}~Section~4. It can also be deduced from T. de Jong's formula for the Euler characteristic of the Milnor fibre (see~\cite{dJ2}), and formulae from de Jong, Pellikaan, Marar and Mond for multi-germs (see Lemma~\ref{Marar-MondAndMuIForMultigerms}). It also can be deduced from~\cite{ST}, Proposition~3.3.  

In what follows, by saying that $f$ is non regular we mean that $f$ is not a mono-germ of embedding.

\begin{prop}\label{propEulerCharFormula}[Siersma, Massey-Siersma] Let $f\colon (\C^2,S)\to(\C^3,0)$ be an $\cA$-finite non regular multi-germ. Then 
\[\chi(\F)=\mu(D)+2C-3T.\]
\end{prop}

The formula does not coincide literally with the expressions stated in~\cite{Si4}~\cite{MaSi}, but it is equivalent to them using Marar-Mond formulae stated below.

\section{Topological invariants}

A key element for the proofs of topological invariances  is contained in work of A. N\'emethi and G. Pint\'er ~\cite{NP} where, among other things, they prove the following: If a finite \emph{mono-germ} has $C<\infty$, then $C$ coincides with the sign-refined Smale invariant of the associated link embedding. As consequences of this, they obtain the $C^{\infty}$ $\cA$-invariance of the number of cross-caps, and they prove the following result:

\begin{prop}\label{propC-3TIsATopInv}If $f\colon (\C^2,0)\to(\C^3,0)$ is an $\cA$-finite mono-germ, then $C-3T$  is invariant by topological $\cA$-equivalence.
\end{prop}

Apart from the results of N\'emethi and Pint\'er, we use the following adaptations for multi-germs of formulae for mono-germs due to Marar and Mond  \cite{MararMondCorank1,Mond:1991} (a previous unpublished mono-germ version of the first formula is attributed to de Jong and Pellikaan \cite{MararMondCorank1}). The first two formulae are used in the proof of Theorem~\ref{thmTopInv}, the third and fourth are included for completeness.

\begin{lem}\label{Marar-MondAndMuIForMultigerms}
Let $f\colon (\C^2,S)\to(\C^3,0)$ be a non-regular $\cA$-finite multi-germ and let $r=\vert S\vert$. Then
\[\mu_I=\mu(D)-\mu(\Sigma)-T,\]
\[\mu(\Sigma)=\frac{1}{2}(\mu(D)-C+2T-r+2).\]
Let $\sigma$ be the number of singular branches of the image of $f$. Let $D^2$ be the closure of the set  $\{(x,x')\in\C^2\times\C^2\mid x\neq x',f(x)=f(x')\}$ and let $D^2/S_2$ be the quotient complex space obtained by identifying $(x,x')$ and $(x',x)$. Then
\[\mu(D^2)=\mu(D)-6T+r(r-2)+\sigma,\]
\[\mu(D^2/S_2)=\mu(\Sigma)-4T+\frac{r(r-1)}{2}+\sigma-1.\]
\begin{proof}We include only the proof of the first formula; the proofs for the remaining ones have similar flavor and can be easily adapted from Marar and Mond original formulae  \cite{MararMondCorank1}.
With notations as in the proof of the Euler characteristic formula $\chi(\F)=\mu(D)+2C-3T$, we have that $\mu_I=\chi(X_s)-1=\chi(X_s\setminus \Sigma_s)+\chi(\Sigma_s)-1=\chi(U_s )-\chi(D_s)+\chi(\Sigma_s)-1$. Now the claim follows from the equalities $\chi(U_s )=r$, $\chi(D_s)=r-\mu(D)+3T$ and $\chi(\Sigma_s)=1-\mu(\Sigma)+2T$.
\end{proof}
\end{lem}

\begin{thm}\label{thmTopInv}
The numbers $\mu_I$, $C$, $T$, $\mu(\Sigma)$ and the collection of $\mu(D,z), z\in S$ depend only on the embedded topological type of $X$. In particular, they are topological $\cA$-invariants of $f$.
\end{thm}

\begin{proof}
Let $f\colon (\C^2,S)\to (\C^3,0)$, $\tilde f\colon (\C^2,\tilde S)\to (\C^3,0)$ be two $\cA$-finite mappings, and let $(X,0)$ and $(\tilde X,0)$ be their images. Assume that there exits a homeomorphism $\psi\colon (\C^3,0)\to (\C^3,0)$ taking $(X,0)$ to $(\tilde X,0)$. Denote by $\Sigma$, $D$, $\tilde \Sigma$ and $\tilde D$ the double point loci of at the target and source of $f$ and $\tilde f$, respectively. The local homeomorphism type of $(X,0)$ at a general point of $\Sigma$ (the union of two planes meeting in a line) is different from the local homeomorphism type of a point of $\tilde X\setminus\tilde \Sigma$, hence we have that $\psi(\Sigma)=\tilde \Sigma$. Observe that $f$ and $\tilde f$ are the normalization mappings of $(X,0)$ and $(\tilde X,0)$. By uniqueness of the topological normalization, there exists a homeomorphism $\phi\colon (\C^2,S)\to (\C^2,\tilde S)$ lifting $\psi$. In other words, the germs $f$ and $\tilde f$ are topologically $\cA$-equivalent. 

The topological $\cA$-equivalence implies that the cardinality of $S$ coincides with the cardinality of $S'$. Since $\psi(\Sigma)=\Sigma'$, we deduce that $\phi(D)=D'$. Being $D$ and $D'$ plane curves, this implies by~\cite{Le73} that the collection of Milnor numbers of the multi-germ of plane curves $(D,S)$ coincides with that of $(D',S')$. 

Let $C$ and $\tilde C$ be the number of crosscaps of $f$ and $\tilde f$. Let $r=\vert S\vert$ and, for $j=1,\dots r$, let $f^{(j)},\tilde f^{(j)}$ and $\phi^{(j)}$ be the branches of $f,\tilde f$ and $\phi$, ordered in a way that $f^{(j)}$  and $\tilde f^{(j)}$ are topologically $\cA$-equivalent via $\phi^{(j)}$ and $\psi$. Observe that all $f^{(j)}$ and $\tilde f^{(j)}$ are $\cA$-finite, hence they have well defined numbers of crosscaps and triple points, respectively $C^{(j)},T^{(j)},\tilde C^{(j)}$ and $\tilde T^{(j)}$. Similarly, let $\F^{(j)},\tilde \F^{(j)},D^{(j)}$ and $\tilde D^{(j)}$ the corresponding Milnor fibres and double point curves. Since $f^{(j)}$ and $\tilde f^{(j)}$ are topologically $\cA$-equivalent, their images have the same embedded topological type. This implies the equality $\chi(\F^{(j)})=\chi(\tilde \F^{(j)})$ by~\cite{Le73}. By the $\cA$-equivalence of $f^{(j)}$ and $\tilde f^{(j)}$ we have that $D^{(j)}$ and $\tilde D^{(j)}$ have the same embedded topological type, and 
hence $\mu(D^{(j)})=\mu(\tilde D^{(j)})$. Comparing both Euler characteristics, we conclude that $2C^{(j)}-3T^{(j)}=2\tilde C^{(j)}-3\tilde T^{(j)}$, by virtue of Proposition~\ref{propEulerCharFormula}. On the other hand, Proposition \ref{propC-3TIsATopInv} states that  $C^{(j)}-3T^{(j)}=\tilde C^{(j)}-3\tilde T^{(j)}$. Both things together imply $C^{(j)}=\tilde C^{(j)}$ and $T^{(j)}=\tilde T^{(j)}$ and, since $C=\sum_{j=1}^{r}C^{(j)}$ and $\tilde C=\sum_{j=1}^{r}\tilde C^{(j)}$, we conclude that $C=\tilde C$.

So far we know that $\vert S\vert$, $\chi(\F)$, $C$ and the collection of Milnor numbers of the multi-germ $(D,S)$  depend only on the embedded topological type of $X$. That the same applies to $T,\mu_I$ and $\mu(\Sigma)$ follows immediately from the formulae $\chi(\F)=\mu(D)+2C-3T$, $\mu(\Sigma)=\frac{1}{2}(\mu(D)-C+2T-r+2)$ and $\mu_I=\mu(D)-\mu(\Sigma)-T$. \end{proof}

As a direct consequence, we obtain the following (see ~\cite{NP} for definitions):

\begin{cor}
The Smale invariant is a topological $\cA$-invariant for those immersions $j\colon S^3\looparrowright S^5$ associated to $\cA$-finite map germs $(\C^2,S)\to(\C^3,0)$. More strongly, it only dependes on the embedded topological type of the image of the immersion. 
\end{cor}


\end{document}